\documentclass[a4paper,11pt]{amsart}
\usepackage{graphicx}
\usepackage{amsmath}
\usepackage{amssymb}
\usepackage{oldgerm}
\usepackage[all]{xy}

\newcommand{\Hom}{\operatorname{Hom}\nolimits}
\renewcommand{\Im}{\operatorname{Im}\nolimits}

\newcommand{\Coker}{\operatorname{Coker}\nolimits}

\newcommand{\id}{\operatorname{id}\nolimits}

\newcommand{\soc}{\operatorname{soc}\nolimits}
\newcommand{\rad}{\operatorname{rad}\nolimits}
\newcommand{\cone}{\operatorname{cone}\nolimits}
\newcommand{\N}{\operatorname{\mathbb{N}}\nolimits}
\newcommand{\Z}{\operatorname{\mathbb{Z}}\nolimits}
\newtheorem{theo}{Theorem}[section]
\newtheorem{cor}[theo]{Corollary}
\newtheorem{lemma}[theo]{Lemma}

\newtheorem{defi}[theo]{Definition}
\newtheorem{rem}[theo]{Remark}
\newtheorem{exa}[theo]{Example}

\begin{document}
\title{Examples of Auslander-Reiten components in the bounded derived Category}
\author{Sarah Scherotzke}
\address{Sarah Scherotzke \newline Mathematical Institute \\ University of Oxford \\ 24-29 St.\
Giles \\ Oxford OX1 3LB \\ United Kingdom}
\email{scherotz@maths.ox.ac.uk}
\date{\today}

 \maketitle

\begin{abstract}
We deduce a necessary condition for Auslander-Reiten components of the bounded derived category of a finite dimensional algebra to have
 Euclidean tree class by classifying certain types of irreducible maps in the category of complexes. This result shows that there are only finitely many Auslander-Reiten components with Euclidean tree class up to shift.
Also the Auslander-Reiten quiver of certain classes of Nakayama are computed directly and it is shown that  they are piecewise hereditary. Finally we state a condition for $\Z[A_{\infty}]$-components to appear in the Auslander-Reiten quiver generalizing a result in \cite{W}.
\end{abstract}

\section{Introduction}
Auslander-Reiten triangles have been introduced by Happel for triangulated categories generalizing
the concept of Auslander-Reiten sequences for finite-dimensional algebras.

In this paper we analyze Auslander-Reiten triangles in the bounded derived category of a finite-dimensional algebra $A$,
denoted by $D^b(A)$. Using Auslander-Reiten triangles, Happel has defined a locally
finite graph, whose vertices correspond to indecomposable
complexes in $D^b(A)$ and whose edges correspond to irreducible maps appearing in
Auslander-Reiten triangles.  The connected components of this quiver are the Auslander-Reiten components.

So far Auslander-Reiten components of $D^b(A)$ for $A$ self-injective
have been determined by Wheeler in \cite{W} and for $A$ a hereditary algebra they have been determined by Happel in \cite[I,5.5]{H1}. But
very little is known about the components of the bounded derived
category of non self-injective algebras. In this paper we compute the Auslander-Reiten quiver of the bounded derived category for a class of algebras with finite global dimension.

In \cite{S} we classified bounded derived categories whose
Auslander-Reiten quiver has either a stable component of Dynkin
tree class, or a stable finite component or a bounded component.
It is shown that a component with Dynkin tree class can only
appear in the Auslander-Reiten component of the bounded derived
category of a piecewise hereditary algebra. In this paper which can be read independently from \cite{S} we will
focus on Auslander-Reiten components with Euclidean tree class.

\smallskip

In the first section we introduce some notation and give the
definition of Auslander-Reiten triangles.  In section two we state
some properties of Auslander-Reiten triangles and show that
certain indecomposable elements lie in Auslander-Reiten components
isomorphic to $\Z[A_{\infty}]$. This generalizes Wheeler's main
result in \cite{W}.

In section three we determine a necessary condition for the existence of stable Auslander-Reiten components with Euclidean tree class.
In the last section we determine the Auslander-Reiten quivers of the bounded derived category for certain classes of algebras and use our results to show that they are piecewise hereditary.

\smallskip

In the third section irreducible maps of complexes in $Comp^{-,b}(\mathcal{P})$ that end in an indecomposable contractible
complex are analyzed.  We show that these irreducible maps start in a complex that is isomorphic in $D^b(A)$ to a simple module of $A$ embedded into $D^b(A)$. Using this result we show that a stable Auslander-Reiten component of $D^b(A)$ which has Euclidean tree class contains a simple $A$-module. It then follows that there are only finitely many Auslander-Reiten components with Euclidean tree class up to shift and this number is bounded by the number of isomorphism classes of simple modules.

\smallskip

Finally in the last section we analyze the Auslander-Reiten quivers of
Nakayama algebras given as path algebras $k A_n/I$ for some ideal $I \trianglelefteq k A_n$.
The first class are such Nakayama algebras with global dimension $n-1$. In this case the Auslander-Reiten quiver consists of only one
component $\Z[A_n]$. So by previous results the algebra $k A_n/I$ is derived equivalent to $A_n$.

The second class are the algebras where $I$ is generated
by the path of length $n-1$. In the second case the Auslander-Reiten quiver consists of only one
component $\Z[D_n]$. Therefore the algebra $ k A_n/I$ is derived equivalent to
$kD_n$.

The notation is the same as in \cite{S}. Also some definitions and results from \cite{S} will be used. For the convenience of the reader the notation and definitions are stated in this paper.

\bigskip

I would like to thank my supervisor Dr Karin Erdmann, Prof. Dieter Happel and Prof. Bernhard Keller for helpful discussions on the topic of this paper.
\section{Preliminaries and Notation}
Let $A$ denote a finite-dimensional indecomposable algebra over a field $k$ and $A$-mod the category of finite-dimensional left $A$-modules. We denote by $\mathcal{P}$ the full subcategory in $A$-mod of projective modules and $\mathcal{I}$ the full subcategory category of injective $A$-modules. Let $C \in \{$A$\mbox{-mod},\  \mathcal{P},\  \mathcal{I}\}$.

Then $Comp^{*,?}(C)$ is the category of complexes that are bounded above if $*=-$, bounded below if $*=+$ and bounded if $*=b$. The homology is bounded if $?=b$.
We denote by $D^b(A)$ the bounded derived category and by $K^{*,?}(C)$ the homotopy category.

The homotopy category and the derived categories are triangulated
categories by \cite[10.2.4, 10.4.3]{We} where the shift functor
$[1]$ is the automorphism. The distinguished triangles are given
up to isomorphism of triangles by
\[ \xymatrix{X\ar[rr]^f&&Y\ar[rr]^{0\oplus {\rm
id}_Y}&&\cone(f)\ar[rr]^{{\rm id}_{X[1]}\oplus 0}&&X[1].}\] for
any morphism $f$.

\medskip

It is difficult to calculate the morphisms in the derived
category of $A$-modules. The following theorem provides an easier way to represent them.

\begin{theo}\cite[10.4.8]{We}
We
have the following equivalences of triangulated categories
\[ \begin{array}{c}
   K^{-,b}(\mathcal{P})\cong K^{+,b}(\mathcal{I})\cong D^b(A).
\end{array}
\]
\end{theo}
We identify a $A$-module $X$ with the complex that has entry $X$ in degree $0$ and
entry $0$ in all other degrees. By abuse of notation we call this
complex $X$. A complex with non-zero entry in only one degree is also called a stalk complex.
Note that $A$-mod is equivalent to a full subcategory of $D^b(A)$ using this embedding.

\smallskip

Let $N$ be a left $A$-module and
$\cdots  \stackrel{d_P^2} \to P_1 \stackrel{d_P^1}\to P_0\to N$ its minimal projective
resolution. Let $N\to I_0 \stackrel{d_I^0} \to
I_1  \stackrel{d_I^1}\to \cdots $ be its minimal injective resolution. Then we denote throughout this paper by $pN$ the complex with $(pN)^i=P_{-i}$
and $d^i:= d_P^{-i}$ for $i \le 0$ and $(pN)^i=0$ for $i>0$.
Similarly we define $iN$ to be the complex with $(iN)^n=I_{n}$ and
$d^n:= d_I^{n}$ for $n \ge 0$ and $(iN)^n=0$ for $n<0$.

\begin{rem}\label{contractible1}
Note that all indecomposable contractible complexes in
$Comp^{-,b}(\mathcal{P})$ have up to shift the form \[ \cdots \to 0 \to P \to_{\id}
P \to 0 \to \cdots
\] for an indecomposable projective module $P$ of $A$. We denote
such a complex where $P$ occurs in degree 0 and $1$ by $\bar P$.
\end{rem}
Contractible complexes are projective objects in $Comp^{-,b}(\mathcal{P})$ (see for example \cite[2.9]{S}).

Finally we define for a complex $X$, the complex $\sigma^{\le n}(X)$ to be the complex with $\sigma^{\le n}(X)^i=X^i$ for $i \le n$ and $d_{\sigma^{\le n}(X)}^i=d^i_X$ for $i <n$ and $\sigma^{\le n}(X)^i=0$ for $i >n$. We define $\sigma^{\ge n}(X)$ analogously.

Next we introduce Auslander-Reiten theory for $D^b(A)$. We state
the existence conditions for Auslander-Reiten triangles in the
bounded derived category of a finite-dimensional algebra and prove
some properties that will be needed in the other sections.

For an introduction to triangulated categories we refer to \cite[1.1]{H}.

\begin{defi}\cite[4.1]{H1} [Auslander-Reiten triangles] \label{triangle}
A distinguished triangle $X\to_u Y  \to_v Z \to_w X[1]$ is called
an Auslander-Reiten triangle if the following conditions are
satisfied:

(1) The objects $X,\ Z$ are indecomposable

(2) The map $w$ is non-zero

(3) If $f:W \to Z$ is not a retraction, then there exists $f':
W\to Y$ such that $v\circ f'=f$.
\end{defi}

By \cite[4.2]{H1} we have the following equivalences. The condition (3) is
equivalent to

(3'') If $f:W \to Z$ is not a retraction, then $w \circ f =0$.

\smallskip

We refer to $w$ as the connecting homomorphism of an
Auslander-Reiten triangle. We say that the Auslander-Reiten
triangle $X \to Y \to Z \to X[1]$ starts in $X$, has middle term
$Y$ and ends in $Z$. Note also that an Auslander-Reiten triangle
is uniquely determined up to isomorphisms of triangles by the
isomorphism class of the element it ends or starts with.  The
Auslander-Reiten translation $\tau$ is the functor from the set of 
isomorphism classes of indecomposable objects that appear at the
end of an Auslander-Reiten triangle to the set of isomorphism classes of indecoposable objects at the start of an Auslander-Reiten triangle. The functor $\tau$ sends the isomorphism
class of $Z$ to  the isomorphism class of $X$.

The Auslander-Reiten translation of $A$-mod will be denoted by $\tau_A$ in this chapter to avoid confusion.

Analogously to the classical
Auslander-Reiten theory we can define irreducible maps, minimal maps,
 left almost split maps and right almost split maps as in \cite[1.1,1.4]{ASS}.
Irreducible maps here have the same properties as in the case of
Artin algebras. (see \cite[1.8,1.10]{ASS})

\begin{lemma}\label{irred}
Let $N, M \in D^b(A)$ and let $f:N \to M$ be an irreducible map in $D^b(A)$.

(1) Let $N \to_g Q \to  E \to TN$ be the Auslander-Reiten triangle,
then there is a retraction $s: Q \to M $ such that $f=s \circ g$.

(2) Let $L \to B \to_h  M \to TL$ be an Auslander-Reiten
triangle, then there is a section $r: N \to B $ such that $f=h
\circ r$.
\end{lemma}

We define the Auslander-Reiten quiver to be the labelled graph
$\Gamma(D^b(A))$ with vertices the isomorphism classes of
indecomposable objects. The label of an arrow $X
\stackrel{(d_{XY}, d'_{XY})} \to Y$ is defined as follows. If
there is an Auslander-Reiten triangle $X \to L \to \tau^{-1}(X)
\to  X[1]$ then $d_{XY}$ is the multiplicity of $Y$ as a direct
summand of $L$. Analogously if there is an Auslander-Reiten
triangle $\tau(Y) \to M \to Y \to \tau(Y)[1]$ ending in $Y$, then
$d'_{XY}$ is the multiplicity of $X$ as a direct summand of $M$.
We call a connected component of $ \Gamma(D^b(A))$ an
Auslander-Reiten component. For well-defindness we refer to
\cite[]{H1}.
\begin{lemma}\label{irred2} \cite [4.3, 4.5]{H1}
Let $X \to_u M \to_v  Z \to X[1]$ be an Auslander-Reiten triangle
and $M \cong M_1\oplus M_2$, where $M_1$ is indecomposable. Let
$i:M_1 \to M$ be an inclusion and $p: M\to M_1$ a projection. Then
$v\circ i: M_1\to Z$ and $p \circ u: X \to M$ are irreducible
maps.  Furthermore $u$ is minimal left almost split and $v$ is
minimal right almost split.
\end{lemma}
Let $\nu_A$ denote the Nakayama functor of $A$ and let \[\nu_A^{-1}:=\Hom_A( \Hom_k(-,k) , A).\]
We denote by $\nu$ the left derived functor of $\nu_A$ on $D^b(A)$ and by $ \nu^{-1} $
the right derived functor of $\nu^{-1}_A$. Then $\nu$ maps a
complex $X\in Comp^{b}(\mathcal{P})$ to $\nu(X) \in Comp^b(\mathcal{I})
$, where $\nu (X)^i:=\nu_A (X^i)$ and
 $d^i_{\nu (X)}:= \nu_A(d^i_X)$ for all $i\in \Z$.

\smallskip

For $D^b(A)$ the existence of Auslander-Reiten triangles have been determined.
\begin{theo}\label{existence}\cite[1.4]{H}
(1) Let $Z \in K^{-,b}(\mathcal{P})$ be indecomposable. Then there
exists an Auslander-Reiten triangle ending in $Z$ if and only if
$Z \in K^{b}(\mathcal{P})$. The triangle is of the form $\nu
(Z)[-1] \to Y \to Z \to \nu(Z)$  for some $Y \in K^{-,b}(\mathcal{P})$.

(2) Let $X \in K^{+,b}(\mathcal{I})$ be indecomposable, then there exists an Auslander-Reiten
triangle starting in $X$ if and only if $X \in
K^{b}(\mathcal{I})$. The triangle is of the form $X\to Y \to \nu^{-1}(X)[1]  \to X[1]$  for some $Y \in K^{-,b}(\mathcal{P})$.
\end{theo}
From this result we deduce that the translation $\tau$ is given by
$\nu[-1]$ and $\tau$ is natural equivalence from
$K^b(\mathcal{P})$ to $K^b(\mathcal{I})$.

Let  $N$, $M  \in D^b(A)$ be two indecomposable elements and
let $f: N \to  M $ be an irreducible map. Then there is an arrow
from $N$ to $M$ in the Auslander-Reiten quiver representing $f$ if
and only if $N \in K^b(\mathcal{I})$ or $M \in K^b(\mathcal{P})$.
\section{Auslander-Reiten triangles}
In this section we establish some properties of Auslander-Reiten
triangles and definitions needed for the rest of this paper.

We call an Auslander-Reiten component $\Lambda$ stable, if $\tau$
is an automorphism on $\Lambda$. By \ref{existence} this is
equivalent to the fact that all vertices in $\Lambda$ are in
$K^b(\mathcal{I})$ and $K^b(\mathcal{P})$.

By \cite[2.2.1]{XZ}and \cite[p.206]{Rie} we have
\begin{cor}
Let $\Lambda$ be a stable Auslander-Reiten component of $D^b(A)$.
Then $\Lambda \cong \Z[T]/I$ where $T$ is a tree and $I$ is an
admissible subgroup of aut$(\Z[T])$.
\end{cor}

The following lemma determines the relation between irreducible
maps, retractions and sections in $K^{-,b}(\mathcal{P})$ and
$Comp^{-,b}(\mathcal{P})$. Note that by duality the same is true
if we replace $K^{-,b}(\mathcal{P})$ by $K^{+,b}(\mathcal{I})$ and
$Comp^{-,b}(\mathcal{P})$ by $Comp^{+,b}(\mathcal{I})$.
\begin{lemma}\label{retraction}\cite[3.1]{S}
Let $ B, C \in Comp^{-,b}(\mathcal{P})$ be complexes
that are not contractible. Let $f:B \to C$ be a map of complexes.

(1) Let $C, B$ be indecomposable. The map $f$ is irreducible in $Comp^{-,b}(\mathcal{P})$ if and only
if $f$ is irreducible in $K^{-,b}(\mathcal{P})$.

(2) Let $C$ be indecomposable. The map $f$ is a retraction in $Comp^{-,b}(\mathcal{P})$ if and only
if $f$ is a retraction in $K(\mathcal{P})$.

(3)Let $B$ be indecomposable. The map $f$ is a section in $Comp^{-,b}(\mathcal{P})$ if and only if
$f$ is a section in $K^{-,b}(\mathcal{P})$.
\end{lemma}
We can therefore choose for an irreducible map in
$K^{-,b}(\mathcal{P})$ an irreducible map in
$Comp^{-,b}(\mathcal{P})$ that represents this map. For the rest
of this chapter all irreducible maps in $K^{-,b}(\mathcal{P})$ or
$K^{+,b}(\mathcal{I})$ will be represented by irreducible maps in
$Comp^{-,b}(\mathcal{P})$ or $Comp^{+,b}(\mathcal{I})$
respectively.

The following straightforward result will be used often to show that a complex is indecomposable.
\begin{lemma}
 Let $X \in Comp(A)$ be such that $X^i$ is indecomposable for all $ i \in \Z$ and $d_X^l =0$ if and only if $X^j=0$ for all $j>l$ or $X^j=0$ for all $j\le l$. Then $X$ is an indecomposable complex.
\end{lemma}
The next result follows from the previous lemma by applying \ref{retraction}.
\begin{cor}\label{indecomposable}
Let $X \in Comp^{-,b}(\mathcal{P})$ ( respectively $Comp^{+,b}(\mathcal{I})$) such that $X^i$ is indecomposable for all $ i \in \Z$ and $d_X^l =0$ if and only if $X^j=0$ for all $j>l$ or $X^j=0$ for all $j\le l$. Then $X$ is an indecomposable object in $K^{-,b}(\mathcal{P})$ (respectively $K^{+,b}(\mathcal{I})$).
\end{cor}
We can determine the homology of the middle term of an Auslander-Reiten triangle ending in the stalk complex of a projective indecomposable module.
\begin{lemma}\label{middle part P}
Let $P$ be a projective indecomposable module and let $M$ be the
middle term of the Auslander-Reiten triangle ending in $P$. Then
$H^1(M)=I/\soc I$ where $I:= \nu(P)$, $H^0(M)= \rad P$ and
$H^i(M)=0$ for all $i \not =0,1$.
\end{lemma}
\begin{proof}Let $I:=\nu_A(P)$ and let $\cdots \to  P_2\to_g  P_1 \to_f P_0$ be a
minimal projective resolution of $I$. Then $I$ is isomorphic to
$pI$ in $D^b(A)$. Let $w: P \to pI$ be a homomorphism in
$Comp^{-,b}(\mathcal{P})$ representing the Auslander-Reiten
triangle ending in $P$. By \ref{existence} the Auslander-Reiten
triangle can be written as $pI [-1] \to \cone(w)[-1] \to P \to_w
pI $. Let $M= \cone(w)[-1]$, then $M$ is given by \[ \cdots \to
P_2 \stackrel{g} \to P\oplus P_1 \stackrel{(f,w^0)}\to P_0 \to 0 \to \cdots \] where
$P \oplus P_1$ appears in degree zero. We have $H^1(M) \cong P_0/
\Im (f,w^0)$ and $\Im (f, w^0)=\Im f+ \Im w^0$. We show next that $\Im (f, w^0) / \Im f $ is simple, hence isomorphic to
$\soc I$. Let $h: P' \to P$ be a projective cover of $\rad P. $
We identify $h$ with the corresponding map of complexes $P' \to P
$. As $h$ is not a retraction we have $w \circ h = 0 $ in
$K^{-,b}(\mathcal{P})$ by \ref{triangle} (3'') and $w \circ h$ is therefore homotopic to
zero. Then there is a map $ s: P' \to P_1$ such that $w^0 \circ h
=f \circ  s$. We visualize this in the following diagram

\[\xymatrix{ \cdots \ar[r] & 0 \ar[r] \ar[d] & P' \ar[r] \ar[d]^{h}\ar[ldd]^{ s}  & 0 \ar[r] \ar[d] & \cdots\\
\cdots \ar[r] & 0 \ar[r] \ar[d] &  P \ar[r] \ar[d]^{w^0} & 0 \ar[r] \ar[d] & \cdots\\
\cdots \ar[r] & P_1 \ar[r]^f & P_0 \ar[r] & 0 \ar[r] & \cdots}\]

As the diagram commutes, we have $w^0(\rad P) \subset \Im f$ and
$w^0( P) \not \subset \Im f $. As $P/\rad P$ is simple, $\Im( f,w^0)/ \Im f$ is also simple and we have
$H^1(M)= I/\soc I$. Furthermore $\ker (f,w^0)=\rad P \oplus \Im
g$. Therefore $H^0(M)= \rad P$. Clearly $H^i(M)=0$ for all $i
\not= 0,1$.
\end{proof}
We generalize Wheeler's construction \cite[2.4]{W} in the next theorem.
\begin{theo}\label{inj proj}Let $P$ be a projective indecomposable module that is not simple and suppose $\nu_A^i (P)$ is injective and projective for all $i \in \Z$.
Then the Auslander-Reiten component $\Lambda$ of $D^b(A)$
containing $P$ is isomorphic to $\Z[A_{\infty}]$.
\end{theo}
\begin{proof} Let $w$ by the connecting
map in the Auslander-Reiten triangle ending in $P$. We can choose $w$ such that $w^0$ induces a projection of $P/\rad P$ onto soc $ \nu_A (P)$ by \ref{middle part P}.
Set $f^i:=\nu_A^i(w^0)$ for all $i \in \Z$. The elements \[P_i:= \cdots \to 0 \to
P \stackrel{f^0} \to \nu_A(P) \to  \cdots \to \nu_A^{i-1}
(P)\stackrel{f^{i-1}} \to \nu_A^i(P) \to 0 \to \cdots \] with entry
$P$ in degree $0$ are complexes by the choice of $w^0$ and they are indecomposable for all $i \in \Z$ by \ref{indecomposable}. We set $P_{-1}=0$. The
Auslander-Reiten triangles in $\Lambda$ are given by \[\nu_A
(P_{i-1})[-1] \to P_{i} \oplus \nu_A(P_{i-2}) [-1]\to P_{i-1}.\] This
can be seen as follows: the middle term of the Auslander-Reiten
sequence ending in $P_{i-1}$ can be taken as the upper row of the next diagram and
the bottom row is the direct summand $P_{i}$.
\[ \xymatrix{ 0 \ar[r] & P  \ar[r]^{(f^0,0)\ \ \ \ \ \ } \ar[d]_{\id}& \nu_A (P)\oplus \nu_A (P) \ar[r]^{\ \ \ \  \ \binom{f^1,0}{0, -f^1}} \ar[d]_
{(\id ,\pm \id)} & \cdots \ar[r] &  \nu_A^{i-1}(P) \oplus \nu_A^{i-1}(P) \ar[r]^{\ \ \ \ \ \binom{-f^{i-1},0}{0,-f^{i-1}}} \ar[d]_{(\id,\id)} &  \nu_A^i(P) \ar[r] \ar[d]_{-\id}& 0 \\
 0 \ar[r] & P \ar[r]_{f^0} \ar@/_/[u]_{\id} & \nu_A (P) \ar[r]_{f^2} \ar@/_/[u]_{\id} & \cdots \ar[r] & \nu_A^{i-1} (P) \ar[r]_{f^{i-1}} \ar@/_/[u]_{\id} & \nu_A^i (P) \ar[r] \ar@/_/[u]_{-\id} &0 } \]
 We have to alternate the signs of the second identity map from column two to $i-1$ so that we have a positive sign in the $i$-th column for all $i \in \Z$. This sequence has as direct summand $P_{i} $ and $\nu_A(P_{i-2})[-1]$. Therefore $\Lambda \cong \Z[A_{\infty}]$, and $\Lambda$ looks as follows:
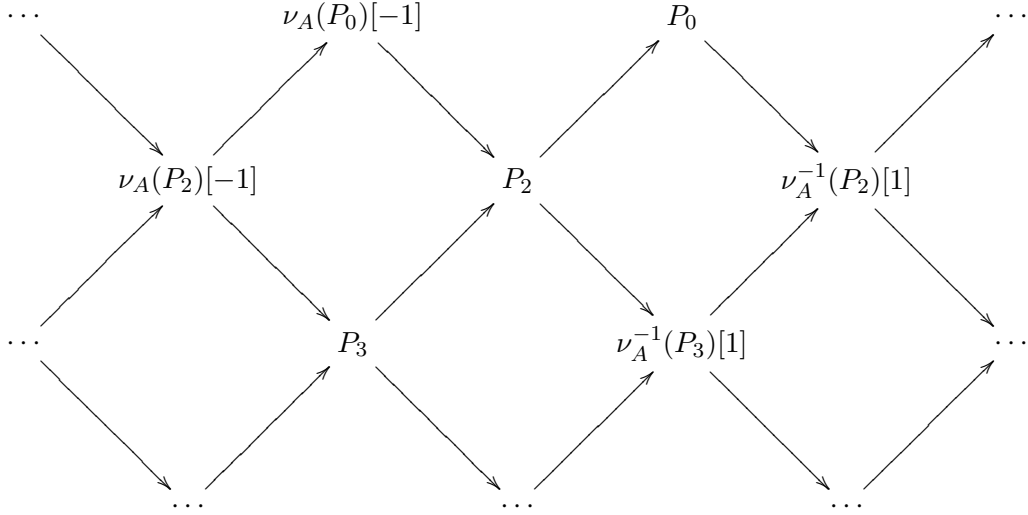
\[\xymatrix@!@=0.5pt{ \cdots \ar[rd] & & \nu_A(P_0)[-1] \ar[rd] &&P_0 \ar[rd] && \cdots \\
& \nu_A (P_2)[-1] \ar[ru] \ar[rd] &&P_2 \ar[ru] \ar[rd] && \nu_A^{-1} (P_2)[1]\ar[ru] \ar[rd]& \\
\cdots \ar[ru] \ar[rd] && P_3\ar[ru] \ar[rd]&&\nu_A^{-1} (P_3)[1]\ar[rd] \ar[ru]&& \cdots\\
& \cdots \ar[ru] && \cdots \ar[ru]&& \cdots\ar[ru]& }\]

As $\nu_A^i(P)[-i] \not \cong P$ the component cannot be a tube.
\end{proof}
For an self-injective algebra we have $\mathcal{P}= \mathcal{I}$
and we can therefore apply the previous theorem to construct
Auslander-Reiten components. We can also use this theorem to compute Auslander-Reiten components of non-self-injective
algebras, as shown in the Example \ref{not self-inj}.

\smallskip

Next we define certain length functions of complexes similar to the functions defined in \cite{W}. 
\begin{defi}
For $Y \in Comp^b(\mathcal{P}) $ we denote by $l(Y)$
the number of projective indecomposable summands in $\oplus_{i\in
\Z}  Y^i$. For $X \in K^b(\mathcal{P})$ we denote by
\[ l_p(X):=\mbox{min} \{ l(Y) | Y \in Comp^b(\mathcal{P})\mbox{ and }Y
\cong X \mbox{ in }K^b(\mathcal{P}). \}\]
\end{defi}
For well-defindness see \cite[4.5]{S}.
\begin{rem}\label{subadditive}
Let $\Lambda$ be a stable component. Then by \ref{existence} every Auslander-Reiten triangle is isomorphic to
\[\tau (P) \to \cone(w)[-1] \to P \stackrel{w} \to \tau (P)[1]\] where
\[0 \to \tau (P) \to \cone(w) [-1] \to P \to 0 \] is a short exact sequence
in $Comp^b(\mathcal{P})$ and $w: P \to \tau(P) [1]$ is a
representative in $Comp^b(\mathcal{P})$ of the connecting
morphism. Then $\cone(w)[-1] \in Comp^b(\mathcal{P})$. Therefore
$l_p$ is well-defined on a stable component and satisfies
$l_p(\cone(w)[-1]) \le l_p(\tau (P)) + l_p(P)$ with equality if
and only $\cone(w)[-1]$ does not have a contractible summand in
$Comp^b(\mathcal{P})$.
\end{rem}
We can generalize \cite[3.2]{W} where the same result is proven in
the case that $A$ is a self-injective algebra. The next result
also generalizes \ref{inj proj} which has only been proven for
stalk complexes of projective modules.
\begin{theo}
Let $X \in Comp^b(\mathcal{P})$ be an indecomposable not contractible complex such that
$\nu_A^i(X^j)$ is projective and injective for all $i,j \in \Z$.
Then $X$ is in an Auslander-Reiten component of $D^b(A)$ isomorphic to $\Z[A_{\infty}]$ or
$A_1$.
\end{theo}
\begin{proof}
Let $C$ be the Auslander-Reiten component of $\Gamma(D^b(A))$ containing $X$ with tree class $T$. The
function $l_p$ is constant on all $\tau$-orbits and therefore $l_p$ is a
subadditive function on $T$. Suppose $l_p$ is bounded. Then by
\cite[4.15]{S} the algebra $A$ is derived equivalent to $D^b(k{Q})$
for a finite Dynkin diagram $ Q \not =A_1$ or $A$ is simple. As
in the first case there is no complex that satisfies the
assumption, we have that $A$ is simple. So $C \cong A_1$.
Otherwise $l_p$ is unbounded and $C$ has therefore tree class
$A_{\infty}$. Assume without loss of generality that $X^i \not =0$
if and only if $0 \le i \le n$. We have $H^0(X) \not =0$ as $X$ does not have contractible summands. Then $\tau(X)$ viewed as complex
in $Comp^b(\mathcal{P})$ satisfies $\tau(X)^i \not =0$ if and only if
$1 \le i \le n+1$. Therefore $X$ is not periodic and $C\cong
\Z[A_{\infty}]$.
\end{proof}
We compute some examples of Auslander-Reiten quivers using the
previous results.
\begin{exa}\label{not self-inj}
Let $k$ be a field and let $G$ be the quiver  \[ \xymatrix{1
\ar@/^/[r]^{\alpha} &2 \ar@/^/[l]^{\beta}}\]  Let $A:=kG/R$ where
$kG$ is the path algebra of $G$, and where $R$ is the ideal
generated by $\{ \alpha \beta \}$. The algebra $A$ has global dimension $2$ and has
finite representation type. We denote by $S_1$ and $S_2$
the simple modules corresponding to the vertices $1$ and $2$. Let
$P_i$ be the projective covers of $S_i$. We have $S_1=P_1/P_2$ and $P_2/S_1= S_2$. We fix a non-zero map $f:P_1
\to P_1$ that maps the generator to the socle. Then $f$ is not an isomorphism.  The derived category $D^b(A)$ has infinitely many
indecomposable elements $I_n$ given by $I_n^m=P_1$ for all $0 \le
m \le n-1$ and $d^l=f$ for $0 \le l \le n-2$ for all $0 \le n \le n-2$.
By \cite[Theorem A]{BGS} we have that $D^b(A)$ is discrete.

As $I_1= P_1$ we have by \ref{inj proj} that all $I_n$ are the
elements of a component $ \Z[A_{\infty}].$

As we have $\tau(S_2)=S_2[1]$ the element $S_2$ belongs to a
component $\Z[A_{\infty}]$ by \cite[4.13]{S} and \cite[4.14]{S}.

Finally we note that the $\tau$-orbit of $P_2$ is given by
$$\tau^n(P_2)^i= \left\{ \begin{array}{lr} P_1, \mbox{ for }-n \le i \le n  \\ P_2, \mbox{ for }i=-n-1 \\
0 \mbox{ else}.\end{array} \right.
$$

Therefore the value of $l_p$ is strictly increasing on
$\tau$-orbits of the component containing $P_2$. The predecessors
of $P_2$ are $S_1$ and $S_1[-1]$. Therefore the component is
$\Z[A_{\infty}^{\infty}]$.
There are no components of Euclidean tree class, as these would
have to contain at least one simple $A$-module by
\ref{euc}.
\end{exa}
The previous example is a discrete derived categories. These categories
have been defined and classified in \cite[Theorem]{V} and their
Auslander-Reiten quivers have been determined in \cite[Theorem
A]{BGS}. The authors do not calculate the Auslander-Reiten quiver
directly but use the fact that the Happel functor (see \cite[2.5]{H}) induces an
equivalence of triangulated categories between $D^b(A)$ and the
stable module category of the repetitive algebra $\widehat{ A}$
for all algebras $A$ of finite-global dimension by \cite[2.3
Theorem]{H}.

\medskip

We can also determine the predecessors for some projective
indecomposable modules $P$ using
 the previous results.
\begin{cor}\label{pro pre}
(1) If $\nu_A (P)$ is projective, then \[ \cdots \to 0 \to P \to
\nu_A (P) \to 0 \to \cdots\] with $P$ appearing in degree $0$, is
the only predecessor of $P$.

(2) If $P$ is injective, then \[ \cdots \to 0 \to P \to \nu_A (P)
\to 0 \to \cdots\] with $P$ appearing in degree $0$ is the only
predecessor of $P$.
 \end{cor}
\begin{proof}
Let $w$ be the connecting homomorphism in the Auslander-Reiten
triangle ending in $P$. Suppose that $\nu_A (P) $ is projective.
Then cone$(w)$ can be taken as $\cdots \to 0 \to P \stackrel{w^0}
\to \nu_A (P)\to 0 \to \cdots$, where $w^0$ induces a projection
of $P/\rad P$ onto soc $ \nu_A (P)$ by \ref{middle part P}. The
complex cone$(w)$ is indecomposable in $K^b(\mathcal{P})$ by
\ref{indecomposable}. This proves part $(1)$ and part $(2)$
follows similarly.
\end{proof}
\section{Irreducible maps ending in contractible complexes}
Next we analyze under which conditions a contractible complex can
appear as direct summand of $\cone(w)[-1] \in
Comp^{-,b}(\mathcal{P})$ for a map $w$ in $Comp^{-,b}(\mathcal{P})$ that induces an
Auslander-Reiten triangle $\nu (Z)[-1] \to Y \to Z \stackrel{w}\to
\nu (Z)$ in $D^b(A)$.

We first introduce a new definition.

\begin{defi}
Let $P_1, \  P_2 \in \mathcal{P}$ and $f:P_1 \to P_2$ be a map.
Then $f$ is $p$-irreducible if $f$ is not a section and not a
retraction and if for any $P \in \mathcal{P}$ and maps $f_1:P_1
\to P$ and $f_2:P \to P_2$ such that $f=f_2 \circ f_1$ we have
that $f_1$ is a section or $f_2$ is a retraction.
\end{defi}
Throughout this section let $P$ be an indecomposable projective
module and $\bar P$ the contractible complex
 \[ \cdots \to 0 \to P \stackrel{\id} \to
P \to 0 \to \cdots \]with ${\bar P}^0={\bar P}^{1}=P$.
\begin{lemma}\label{homotopic middle}
Let $f:Q \to \bar P$ be an irreducible map in $Comp^{-,b}(\mathcal{P})$, where $Q\in
Comp^{-,b}(\mathcal{P})$ is indecomposable and not contractible. Then
there exists an indecomposable projective module $P_0$ and a map $d:P_0 \to P$ that is $p$-irreducible, such
that $Q\cong p(\Coker(d))$.
\end{lemma}

\begin{proof}
Let $f$ be an irreducible map given by the diagram
\[ \xymatrix{ \cdots  \ar[r] & Q^{-1} \ar[r]\ar[d] & Q^0 \ar[r]_{d}
\ar[d]_{f^0} & Q^1 \ar[r]\ar[d]_{f^1} & Q^2 \ar[r]\ar[d] & \cdots
\\ \cdots \ar[r] & 0 \ar[r] & P \ar[r]_{\id}
 & P \ar[r] & 0 \ar[r] & \cdots }\]
We first show that $Q^i=0$ for $i \ge 2$, and that $Q^1\cong P$.
  We can factorize $f$ through $\sigma^{\le 1} Q$ as

  \[\xymatrix{ \cdots \ar[r] & Q^{-1} \ar[r]\ar[d]_{\id} & Q^0 \ar[r]_{d}
\ar[d]_{\id} & Q^1 \ar[r]\ar[d]_{\id} & Q^2 \ar[r]\ar[d] & \cdots
\\
\cdots \ar[r] & Q^{-1} \ar[r]\ar[d] & Q^0 \ar[r]_{d} \ar[d]_{f^0}
& Q^1 \ar[r]\ar[d]_{f^1} & 0 \ar[r]\ar[d] & \cdots \\
\cdots \ar[r] & 0 \ar[r] & P \ar[r]_{\id}
 & P \ar[r]& 0 \ar[r] & \cdots}\]

The map given by the last two rows is not a
retraction, as $f$ is not a retraction. Therefore the map between the first two rows is a section and $Q^i=0$
for all $i\ge 2$.

 We can factorize $f$ as \[\xymatrix{ \cdots \ar[r] & Q^{-1} \ar[r]\ar[d] & Q^0 \ar[r]_{d}
\ar[d]_{d} & Q^1 \ar[r]\ar[d]_{\id} & 0 \ar[r]\ar[d] & \cdots
\\ \cdots \ar[r] & 0 \ar[r]\ar[d] & Q^1 \ar[r]_{\id} \ar[d]_{f^1} &
Q^1 \ar[r]\ar[d]_{f^1} & 0 \ar[r]\ar[d] & \cdots
\\ \cdots \ar[r] & 0 \ar[r] & P \ar[r]_{\id}
 & P \ar[r]& 0 \ar[r] & \cdots}\]

If the map between the first two rows is a section, then $Q$ is isomorphic to the complex in the middle row, and hence is
contractible, which is a contradiction. Therefore the map between
the last two rows is a retraction. We can write $Q^1 \cong P
\oplus P'$ for a projective module $P'$ and $f^1$ is a retraction.
But then we can factorize $f$ as follows

\[\xymatrix{ \cdots \ar[r] & Q^{-1} \ar[r]\ar[d]_{\id} & Q^0 \ar[r]_{d}
\ar[d]_{\id} & P \oplus P' \ar[r]\ar[d]_{f^1} & 0 \ar[r]\ar[d] &
\cdots
\\ \cdots \ar[r] & Q^{-1} \ar[r]\ar[d] & Q^0 \ar[r]_{f^1 d} \ar[d]_{f^0} &
P \ar[r]\ar[d]_{\id} & 0 \ar[r]\ar[d] & \cdots
\\ \cdots \ar[r] & 0 \ar[r] & P \ar[r]_{\id}
 & P \ar[r]& 0 \ar[r] & \cdots}\]

If the map in the last two rows is a retraction, then $f$ is a retraction. Therefore the map between the two upper rows has to be a
section and $P'=0$, $f^0=d$ and $f^1=\id$.

Let $\cdots  \to L_{-2} \to L_{-1} \to \ker d$ be a minimal
projective resolution of $\ker d$. Then $f$ factorizes through
$\Coker (d)$ as follows

\[\xymatrix{ \cdots \ar[r] & Q^{-1} \ar[r]_{h} \ar[d] & Q^0 \ar[r]_{d}
\ar[d]_{\id} & P \ar[r]\ar[d]_{\id} & 0 \ar[r]\ar[d] & \cdots
\\ \cdots \ar[r] & L_{-1} \ar[r]_g \ar[d]& Q^0 \ar[r]_{d} \ar[d]_{d}
& P \ar[r]\ar[d]_{\id} & 0 \ar[r]\ar[d] & \cdots
\\ \cdots \ar[r] & 0 \ar[r] & P \ar[r]_{\id}
 & P \ar[r]& 0 \ar[r] & \cdots}\]

Clearly the bottom diagram is not a retraction, because else $f$
would be a retraction. Therefore the upper diagram is a section
and the $Q^{-r} $ are isomorphic to direct summands of $L_{-r}$
for all $r>0$. Then there exists a projection $s: L_{-1} \to
Q^{-1}$ such that $hs=g$. By the minimality of the projective
resolution, we have $Q^{-1} \cong L_{-1}$. By induction $Q$ is
isomorphic to the complex consisting of a minimal projective
resolution of $\Coker d$ and 0 elsewhere. Suppose $d$ is not
$p$-irreducible. Then we can factorize $d=s\circ t$ where
$t:Q^0\to  \tilde P$ is not a section and $s:\tilde P\to P$ is not
a retraction for some projective module $\tilde P$.
 But then we can factorize $f$ as follows

\[\xymatrix{ \cdots \ar[r] & Q^{-1} \ar[r]_h\ar[d]_{\id} & Q^0 \ar[r]_{d}
\ar[d]_{t} & P \ar[r]\ar[d]_{\id} & 0 \ar[r]\ar[d] & \cdots
\\ \cdots \ar[r] & Q^{-1} \ar[r]_{th} \ar[d]& \tilde P \ar[r]_{s} \ar[d]_{s}
& P \ar[r]\ar[d]_{\id} & 0 \ar[r]\ar[d] & \cdots
\\ \cdots \ar[r] & 0 \ar[r] & P \ar[r]_{\id}
 & P \ar[r]& 0 \ar[r] & \cdots}\]
This is a contradiction to the irreducibility of $f$ as the map
between the first two rows is not a section and the map between
the two bottom rows is not a retraction. Therefore $d$ is
$p$-irreducible.
\end{proof}
Next we determine $p$-irreducible maps between indecomposable projective modules.

\begin{lemma}\label{p-irred} Let $P_1$ and $P_2$ be indecomposable projective modules and $d:P_1 \to P_2$ some map. Then $d$ is $p$-irreducible if and only if $P_1$ is isomorphic to a direct summand of the projective cover of $\rad P_2$ and $d$ is induced by this projection.
\end{lemma}
\begin{proof}
$\Longleftarrow$

Let $d: P_1  \to P_2$ be as in the statement. Suppose that $d$ factors as $d=g \circ f$ where
$f:P_1 \to \tilde P$, $g: \tilde P \to P_2$ and $g$ is not a
retraction. Then $ \Im d \subset \Im g\subset \rad P_2 $. Let $
\pi: P_1 \oplus P_1' \to \rad P_2$ be the projective cover. Since $\tilde P$ is projective, there is a map $e:\tilde P \to P_1 \oplus P_1'$ such that $\pi e=g$. Then we have a commutative diagram

\[ \xymatrix{P_1 \ar[r]_f \ar[d]_d &\tilde P \ar[d]_g \ar[r]_e & P_1 \oplus P_1'\ar[d] \\
\Im d \ar[r] & \Im g \ar[r] & \rad P_2 }\]

Let top $P_1 \cong S_1$ and let $q: \rad P_2 \to S_1$ be a projection with $q(\Im d) \not =0$. Then $q(\Im ( e f)) \not =0$. Let $p: P_1 \oplus P_1' \to P_1$ denote the natural projection. Then $p e f \not \subset \rad P_1$. Therefore $p e f$ is an isomorphism.
So $f$ is a section and $d$ is $p$-irreducible.

$\Longrightarrow$

Conversely, let $s: P_1 \to P_2$ be a $p$-irreducible map. Then $s$ is not
surjective and therefore $\Im s \subset \rad P_2$. Let $P_0$ be a projective module and $i: P_0 \to \rad P_2$ be a projection. Then there is a map $h: P_1 \to P_0$ such that
$s=ih$. As $s$ is $p$-irreducible, $h$ is a
section and $P_1$ is a direct summand of $P_0$.
\end{proof}
Next we determine which $p$-irreducible maps induce an irreducible map in $Comp^{-,b}(\mathcal{P})$.

\begin{theo}\label{contr irred1}
Let $P_0$ be a projective module and let $d:P_0 \to P$ be a
$p$-irreducible map. Then there is an irreducible map in
$Comp^{-,b}(\mathcal{P})$ from the complex $p(\Coker(d))$ to the
contractible complex $\bar P$ if and only if $P_0$ is the
projective cover of $\rad P$ and $d$ is the projection onto $\rad P$.
\end{theo}

\begin{proof}
$\Longleftarrow$

Let $p(\Coker(d)):= \cdots \to P_{-1} \to P_0 \stackrel{d} \to P
\to 0 \to \cdots $, where $P_0$ is the projective cover of $\rad
P$. We claim that the map $h$ given by

\[ \xymatrix{ \cdots
\ar[r] & P_{-1} \ar[r]\ar[d] & P_0 \ar[r]_{d}
\ar[d]_{d} & P \ar[r]\ar[d]_{\id} & 0 \ar[r]\ar[d] & \cdots
\\ \cdots \ar[r] & 0 \ar[r] & P \ar[r]_{\id}
 & P \ar[r]& 0 \ar[r] & \cdots } \]
is irreducible in $Comp^{-,b}(\mathcal{P})$.

 Suppose the given map factors through a complex $X \in Comp^{-,b}(\mathcal{P})$ via maps $f: p(\Coker(d))\to X$ and $g:X\to \bar P$ such that $g$ is not a retraction.
 As $\id_P= g^1 \circ f^1$, we have that $f^1$ is a section and $g^1$ is a retraction. We have that $X^1 \cong P\oplus \tilde P $ for some projective module $\tilde P$. If $ \Im g^1 d_X^0= P$, then $g$ is a retraction.
 Therefore $\Im g^0= \rad P$, as $\Im d \subset \Im g^1 d_X^0= \Im g^0$. Then $d$ factors through $g^0$ and $f^0$. As $d$ is irreducible by assumption, $f^0$ is a section. Therefore $X^0=P_0\oplus P'$ for some projective module $P'$ and $g^1 d_X^0(P')=g^0(P')=0$. We visualize this in the next diagram

 \[ \xymatrix{ \cdots
\ar[r] & P_{-1} \ar[r]\ar[d] & P_0 \ar[r]_{d}
\ar[d]_{\binom{\id}{0} } & P \ar[r]\ar[d]_{\binom{\id}{0}} & 0 \ar[r]\ar[d] & \cdots
\\
\cdots \ar[r] & X^{-1} \ar[r]\ar[d] & P_0 \oplus
P'\ar[r]_{\binom{d,0}{0,x}}
\ar[d]_{(d,0)} & P \oplus \tilde P  \ar[r]\ar[d]_{(\id,0)} & 0 \ar[r]\ar[d] & \cdots \\
\cdots \ar[r] & 0 \ar[r] & P \ar[r]_{\id}
 & P \ar[r]& 0 \ar[r] & \cdots } \]

Using the fact that $ \cdots \to P_{-1} \to \ker d $ is a minimal
projective resolution of $\ker d$, we can factorize the map $g$
through $s: X\to p \Coker(d)$ and $t:p \Coker (d) \to \bar P$ as
follows
\[ \xymatrix{\cdots \ar[r] & X^{-1} \ar[r]\ar[d]& P_0 \oplus
P'\ar[r]_{\binom{d,0}{0,x}}
\ar[d]_{(\id ,0)} & P \oplus \tilde P  \ar[r]\ar[d]_{(\id,0)} & 0 \ar[r]\ar[d] & \cdots \\
\cdots \ar[r] & P_{-1} \ar[r]\ar[d] & P_0 \ar[r]_{d}
\ar[d]_{d} & P  \ar[r]\ar[d]_{\id} & 0 \ar[r]\ar[d] & \cdots \\
\cdots \ar[r] & 0 \ar[r] & P \ar[r]_{\id}
 & P \ar[r]& 0 \ar[r] & \cdots } \]
So $s \circ f $ is an
isomorphism. Therefore $f$ is a section.

$\Longrightarrow$

Conversely, if there is an irreducible map $p (\Coker(d)) \to \bar P$, then $P_0$ is a direct summand of the projective cover of $\rad P$ and $d$ is the induced map by \ref{homotopic middle} and \ref{p-irred}. Let $r:\tilde P \to \rad P$ be the projective cover of $\rad P$ and suppose $\tilde P \cong P_0 \oplus P'$ where $P' \not =0$ is a projective module and let $i:  P_0 \to \tilde P$ be a section such that $r \circ i=d$.  Let $p(\Coker (r)) := \cdots \to S_{-1} \to \tilde P \stackrel{r} \to P
\to 0 \to \cdots $. Then the map of complexes $ p (\Coker (d)) \to \bar P$ induced by the $p$-irreducible map $d$ as in the first part of the proof factors through $p(\Coker (r))$ as follows:

 \[ \xymatrix{ \cdots
\ar[r] & P_{-1} \ar[r]\ar[d] & P_0 \ar[r]_{d}
\ar[d]_{i } & P \ar[r]\ar[d]_{\id} & 0 \ar[r]\ar[d] & \cdots
\\
\cdots \ar[r] & S_{-1} \ar[r]\ar[d] & \tilde P \ar[r]_{r}
\ar[d]_r & P  \ar[r]\ar[d]_{\id} & 0 \ar[r] \ar[d] & \cdots \\
\cdots \ar[r] & 0 \ar[r] & P \ar[r]_{\id}
 & P \ar[r]& 0 \ar[r] & \cdots } \]
The map between the first two rows is not a section and the map between the last two rows is not a retraction. This shows that the map is not irreducible in $Comp^{-,b}(\mathcal{P}).$ Therefore $P_0 \cong \tilde P$.
\end{proof}

Let $P_1$ be a summand of the projective cover of $\rad P $ and
$d: P_1 \to P$ the induced map. Then $p(\Coker (d) ) \in
K^{-,b}(\mathcal{P})$ is isomorphic to the stalk complex $P/ \Im d \in
D^b(A)$ and is therefore indecomposable.

\begin{cor}\label{simple der}Let $\theta$ be a stable Auslander-Reiten component of $\Gamma(D^b(A))$. Then $l_p$ is not additive if and only if there exists an indecomposable projective module $P$ such that
$P/ \rad P$ has finite projective and finite injective dimension and a shift of the stalk complex $P/ \rad P$ is in $\theta$.
\end{cor}
\begin{proof} By \ref{subadditive} the function $l_p$ is not additive if and only if there exists an indecomposable complex  $L \in Comp^b(\mathcal{P})$ and a map $w: L\to \nu (L)$ in $Comp^{b}(\mathcal{P})$ that represents an Auslander-Reiten
triangle in  $\theta$ such that $\cone(w)[-1]\in Comp^b(\mathcal{P})$ contains a contractible summand.

By \ref{retraction} there is a complex $\bar P$ that is a direct summand of $\cone
(w)[-1]$ and an irreducible map $f:p \nu (L)[-1] \to \bar P$ in $Comp^{-,b}(\mathcal{P})$. By
Lemma \ref{p-irred}, $\nu (L)[-1]$ is isomorphic in $D^b(A)$ to
$P/ \rad P$ up to shift. As $\theta$ is a stable Auslander-Reiten component,
we have $p(P/\rad P ) \in K^b(\mathcal{P})$ and $i(P/ \rad P ) \in
K^b(\mathcal{I})$. Therefore $P/ \rad P$ has finite projective and
finite injective dimension.
\end{proof}
From this corollary it follows that if $A$ is self-injective, then
$l_p$ is an additive function on the Auslander-Reiten components
of $D^b(A)$.

We can deduce that Euclidean components always contain the stalk complex of a simple module.
\begin{theo}\label{euc}
Let $C$ be a stable component of the Auslander-Reiten quiver of
$D^b(A)$ with tree class an Euclidean diagram.
Then $C$ contains a simple module.
\end{theo}

\begin{proof}
The function $l_p$ is subadditive on $C$. Suppose $l_p$ is
additive. By \cite[2.4]{Web}, the function $l_p$ takes bounded values
on $C$. This is a contradiction to \cite[4.14]{S}. Therefore $l_p$
is not additive. By \ref{simple der}, this means that $C$ contains a simple
module.
\end{proof}
As there are only finitely many isomorphism classes of simple modules, we get
\begin{cor}
The Auslander-Reiten quiver of $D^b(A)$ contains finitely many Auslander-Reiten components up to shift
that have Euclidean tree class. Their number is bound by the number of isomorphism classes of simple $A$-modules.
\end{cor}
\begin{lemma}[Irreducible maps that do not appear in Auslander-Reiten triangles]
Let $f:B \to C$ be an irreducible map in $D^b(A)$ that does not appear in an Auslander-Reiten triangle. Then $B, C \not \in K^{b}(\mathcal{P})$ and $ B, C \not \in K^b(\mathcal{I})$.
\end{lemma}
\begin{proof}
By \ref{existence} it is clear that $B \not \in K^b(\mathcal{I})$
and $C \not \in K^b(\mathcal{P})$. Let us assume that $B \in
K^b(\mathcal{P})$ and let $n\in \N$ be minimal such that $B^n \not
=0$. Then $f$ factorizes through $\sigma^{\ge n-1}(C)$, where $C$
is represented as a complex in $Comp^{-,b}(\mathcal{P})$. Let $f=h
\circ g$ be this factorization, then $g$ is not a section, as $f$
is not a section and $h$ is not a retraction as $ \sigma^{\ge
n-1}(C) \not \cong C$. This is a contradiction to the fact that
$f$ is irreducible. Therefore $B  \not \in K^b(\mathcal{P})$.
Analogously, we can show that $C \not \in K^b(\mathcal{I})$.
\end{proof}
An example of such an irreducible map not appearing in the Auslander-Reiten quiver of $D^b(A)$ is given at the end of \cite{S}.

\section{Auslander-Reiten triangles of Nakayama algebras}
In this section we analyze the Auslander-Reiten quiver of certain
Nakayama algebras $A$ with finite global dimension. The results at the end of this section may help to illustrate
results from the previous sections.

Let \[A_n: 1 \to 2 \to \cdots \to n-1 \to n \] and let $I$ be an
ideal of the path algebra of $k A_n$. We define $A:= kA_n/I$, then
$A$ is a Nakayama algebra of finite global dimension. We denote by
$J$ the ideal generated by the paths of length one in $kA_n$.

We denote the starting vertex of a path $b$ by $s(b)$ and the vertex that $b$ ends in by $t(b)$. With this notation we have $s(n-1,n)=n$ in the above diagram.
Then $s(b)\ge t(b)$ and $b_1 b \not= 0$ for two paths $b_1$ and $b$ if $t(b)=s(b_1)$.

We denote by $S_i$ the simple modules at the vertex $i$, by $P_i$ the
projective indecomposable modules with $P_i/ \rad P_i =S_i$ and by
$I_i$ the injective indecomposable modules with $\soc I_i=S_i$.
Whenever $S_j$ occurs in $I_i$ then it occurs precisely once,
therefore $\dim_k \Hom_A(I_i, I_j) \le 1$. So we can fix maps
$d_i^j:I_i \to I_j$ which are non-zero if and only if $\dim_k
\Hom_A(I_i, I_j) =1$ such that $d_i^i= \id$ and $d_k^j \circ
d_i^k=d_i^j $ for all $1 \le i \le k \le j \le n$. When
constructing injective resolutions we can always take these maps.
Analogously, when we construct minimal projective resolutions we
can always take maps $\nu_A^{-1}(d_i^j)$.

\smallskip

We investigate the global dimension of Nakayama algebras with
quiver $A_n$. We call two paths intersecting, if they have at
least one arrow in common. Note that we can always choose a
generating set $B$ of an ideal $I$ in $kA_n$ that consists of paths
in $k A_n$. If the generating set $B$  is minimal there are no
two paths in $B$ that contain each other. We call a subset of $B$ intersecting, if every path in that
set intersects with at least one other path in the set.
\begin{lemma}\label{gdim}
Let $A= kA_n/I$. Let $B$ be a minimal generating set of $I$ that consists of paths in $kA_n$.
Then the global dimension of $A$ is smaller or equal to the maximal cardinality of all intersecting subsets of $B$ plus one.
\end{lemma}
\begin{proof}
We denote the starting vertex of $b_i$ by $s_i:=s(b_i)$ and the vertex that $b_i$ ends in by $t_i:=t(b_i)$. We assume that paths generating $B$ are ordered such that $s_i \le s_j$ if $i < j$. For a fix $ i$ we consider $\{ b_{m}, \cdots ,b_{i} \}$, the intersecting set of $B$ that does not contain elements $b_j$ for $j >i$ and is maximal with that property.

Suppose first that for all $m < j < i $ there is an element $s_w$ for some $m \le  w  \le i$ with $s_w \in [t_{j}, t_{j+1}]$

Then a minimal projective resolution of $S_{s_i}$ is of the form

\[ 0\to P_{t_{m}} \to \cdots  \to P_{t_{i-1}} \to P_{t_i} \to P_{s_i-1} \to  P_{s_i}.\]Therefore $S_{s_i}$ has projective dimension $i-m+2$.

Suppose there is some $j$ with $m <  j < i $ such that the interval
$[t_{j}, t_{j+1}]$ does not contain any $s_w$ for all $m \le w \le
i$. Let $j$ be maximal with this property, then a minimal projective
resolution of $S_{s_i}$ is given by

\[ 0\to P_{t_j} \to \cdots  \to P_{t_{i-1}} \to P_{t_i} \to P_{s_i-1} \to  P_{s_i}.\]

Therefore $S_{s_i}$ has projective dimension $i-j+2$. As $j\ge m+1$, we have that the projective dimension is
$\le i-m+1$.

If $j \not = s_l$ for all $b_l \in B$, then a minimal
projective resolution of $S_j$ is given by $ 0 \to P_{j-1} \to P_j
$ and $S_j$ has therefore projective dimension 1.

We can use the same argument for injective resolutions. This proves the lemma.
\end{proof}

\begin{cor}
The algebra $A$ has global dimension $n-1$ if and only if $I$ is generated by all paths of length 2.
\end{cor}
\begin{proof}
If $I$ is generated by all paths of length 2, then we are in the
first case of \ref{gdim} and $S_i$ has projective dimension $i-1$ for all $1 \le i \le n$.
Suppose $A$ has global dimension $n-1$, then $S_n$ has projective dimension $n-1$. So we are in the first case of \ref{gdim}.
Therefore a maximal intersecting set of $B$ is given by $\{b_1, \cdots ,b_{n-2}\}$ with ${t_j}=j$ and $s_{n-2}=n$ for all $1 \le j \le n-1$.
Then the $b_j$ have to be paths of length two. Furthermore $I$ is generated by the $b_j$ as any other path in $A_n$ of lenght $>1$ can be written
as a product of the $b_j$.
\end{proof}
By \cite[3.5]{ASS} every indecomposable module $M$ of $A$ is
isomorphic to $P_i/\rad^t(P_i)$. If $M$ is projective then we
choose $t$ minimal such that $\rad^t(P_i)=0$. Then $i$ and $t$ are
uniquely determined. We fix $i$ and $t$. Let $l:=l(P_i)$ and $
\bar l :=l (I_{i-t+1})$. If pdim$M \ge 1$, then the minimal
projective resolution of $M$ starts with $ \cdots \to P_{i-l} \to
P_{i-t} \to P_i$. If idim$M\ge 1$ then the minimal injective
resolution of $M$ starts with $ I_{i-t+1}  \to I_{i+1} \to I_{i+
\bar l-t +1}\to \cdots $. We thereby set $I_k=0$ and $P_k=0$ if
$k\le 0$.

\medskip

We introduce three conditions:

(1) We have $d_{i-l}^{i-t+1}=0$ or pdim$M \le 1$.

(2) We have $d_{i}^{i+ \bar l -t +1}=0$ or idim$M\le 1$.

(3) We have $d_{i-t}^{i+1}=0$ or pdim$M=0$ or idim$M=0$.

\smallskip

The next lemma will be used to calculate concrete examples.
\begin{lemma}\label{pre}
Let $w:M \to \nu(M)$ define an Auslander-Reiten triangle
terminating in $M$.

(a) $\cone(w)$ has a direct summand
 isomorphic to $ \nu (\Omega(M))[1]$ if and only if (1) and (3) hold.

(b) $\cone(w)$ has a direct summand
isomorphic to $\Omega^{-1}(M)$ if and only if (2) and (3) hold.

(c) $\cone(w)$ has a direct summand isomorphic to
\[ \cdots \to 0 \to I_{i-t+1} \to I_i \to 0 \cdots\]if and only if (1) and (2) hold.

(d)  $\cone(w)$ decomposes as sum of the indecomposable complexes isomorphic to
$\Omega^{-1}(M)$, $\cdots \to 0 \to I_{i-t+1} \to I_i \to 0
\cdots$, and $ \nu (\Omega(M))[1]$ if and only if (1), (2) and (3) hold.

(e) If at most one condition (1)-(3) holds, then $\cone(w)$ is indecomposable.
\end{lemma}
\begin{proof}
We will write $d$ instead of $d_i^j$ for an easier presentation.
The connecting map $w$ of the Auslander-Reiten triangle ending in
$M$ can be taken as a map in $Comp^{+,b}(\mathcal{I})$
\[ \xymatrix{\cdots \ar[r] & 0\ar[r] \ar[d]& 0\ar[d] \ar[r] & I_{i-t+1} \ar[d]_d \ar[r]_d & I_{i+1}
\ar[r]_d \ar[d]&I_{i+ \bar l-t+1} \ar[d]\ar[r]& \cdots \\
\cdots \ar[r] & I_{i-l} \ar[r]_d & I_{i-t} \ar[r]_d  & I_i \ar[r]
&0 \ar[r]& 0\ar[r] & \cdots}\] Note that if pdim$M=0$, then
$I_{i-t}$ does not appear and all entries left of $I_{i-t}$ are
zero. If pdim$M=1$, then all entries left of $I_{i-t} $ are zero.
Analogously if idim$M=0$, then $I_{i+1}$ does not appear and all
entries right of $I_{i+1} $ are zero. If idim$M=1$, then all
entries right of $I_{i+1} $ are zero. Then cone$(w)$ is given by
the complex

\[ \xymatrix{ \cdots \ar[r] & I_{i-l} \ar[r]_{\binom{d}{0}}& I_{i-t+1}
\oplus I_{i-t}  \ar[r]_{\binom{-d,0}{d,d}} & I_{i+1} \oplus I_{i}
\ar[r]_{(d,0) } & I_{i+ \bar l -t+1} \ar[r] & \cdots }\]

We determine the direct summands of $\cone(w)$ in $Comp^{+,b}(\mathcal{I})$ as by \ref{retraction}
the direct sum decomposition in $Comp^{+,b}(\mathcal{I})$ gives the direct sum decomposition in
$D^b(A)$.

We have a surjective map $f$ given by

 \[ \xymatrix{ \cdots \ar[r] & I_{i-l}\ar[d]_{\id}  \ar[r]_{\binom{d}{0}}& I_{i-t+1}
\oplus I_{i-t}  \ar[r]_{\binom{-d,0}{d,d}} \ar[d]_{(0, \id)} & I_{i+1} \oplus I_{i}
\ar[r]_{(d,0) }\ar[d] & I_{i+ \bar l -t+1} \ar[r] \ar[d] & \cdots\\ \cdots \ar[r] &
I_{i-l} \ar[r]_d & I_{i-t} \ar[r]& 0 \ar[r]
 & 0  \ar[r] & \cdots}\]
Note that up to multiplication with a non-zero scalar, this is the only surjective map between the two complexes.
If there is a non-zero  map $g$ from the bottom complex to the top complex such that $f \circ g=\id$, then it is of the following form
\[ \xymatrix{ \cdots \ar[r] & I_{i-l} \ar[r]_{\binom{d}{0}}& I_{i-t+1}
\oplus I_{i-t}  \ar[r]_{\binom{-d,0}{d,d}} & I_{i+1} \oplus I_{i}
\ar[r]_{(d,0) } & I_{i+ \bar l -t+1} \ar[r] & \cdots\\ \cdots
\ar[r] & I_{i-l}\ar@/_/[u]_{\id}  \ar[r]_d & I_{i-t}
\ar@/_/[u]_{\binom{\lambda d}{\id}} \ar[r]& 0 \ar[r] \ar@/_/[u] &
0 \ar@/_/[u] \ar[r] &  \cdots }\] where $\lambda \in k$. We need to determine under which conditions $g$ is a chain map.
As the middle square commutes, we have $(-\lambda+1) d_{i-t}^i=0$ and
$\lambda d_{i-t}^{i+1}=0$. As the square on the left hand side
commutes, we have $\lambda d_{i-l}^{i-t+1}=0$. If pdim$M >0$,
then $d_{i-t}^{i} \not = 0$ as $\nu_A^{-1}(d_{i-t}^i)$ appears in a
minimal  projective resolution of $M$. Therefore $\lambda = -1$.
Then $ d_{i-t}^{i+1}=0$ or idim$M=0$ as in this case $I_{i+1}$
does not appear in the diagram. This is then equivalent to
condition (2). Furthermore we need $d_{i-l}^{i-t+1}=0$ or pdim$M=0$
which is equivalent to condition (1). If pdim$M=0$, then
$I_{i-t}$ and $I_{i-l}$ do not appear and therefore the diagram
commutes.

The bottom row is isomorphic to $ \nu (\Omega(M))[1]$ in $D^b(A)$. This is an indecomposable complex, as $\Omega(M)$ is indecomposable. Therefore $\cone(w) $ has an indecomposable summand isomorphic to $\nu (\Omega(M))[1]$.

\medskip

We have an injective map $s$ given by
\[ \xymatrix{ \cdots \ar[r] & I_{i-l} \ar[r]_{\binom{d}{0}}& I_{i-t+1}
\oplus I_{i-t}  \ar[r]_{\binom{-d,0}{d,d}} & I_{i+1} \oplus I_{i}
\ar[r]_{(d,0) } & I_{i+ \bar l -t+1} \ar[r] & \cdots\\
\cdots \ar[r] & 0 \ar[r] \ar@/_/[u]& 0 \ar@/_/[u] \ar[r]& I_{i+1}
\ar[r] \ar@/_/[u]_{\id} & I_{i+ l -t+1} \ar@/_/[u]_{\id} \ar[r] &
\cdots }\] The map $s$ is up to multiplication with non-zero
scalar the only injective map from the bottom row to the top row.

Suppose there is a map $h$ such that $h \circ s= \id$. Then $h$ is
of the form
\[ \xymatrix{ \cdots \ar[r] & I_{i-l} \ar[r]_{\binom{d}{0}}\ar[d] & I_{i-t+1}
\oplus I_{i-t} \ar[d]  \ar[r]_{\binom{-d,0}{d,d}} & I_{i+1} \oplus
I_{i} \ar[r]_{(d,0) }\ar[d]_{(\id,\lambda d)} & I_{i+ \bar l -t+1}
\ar[r] \ar[d]_{\id}& \cdots\\ \cdots \ar[r] & 0 \ar[r] & 0 \ar[r]&
I_{i+1} \ar[r] & I_{i+ l -t+1} \ar[r] & \cdots}\] where $\Lambda
\in k$. As in the first case, we need to determine under which conditions $h$ is a chain map. The middle square commutes if and only if $(\lambda-1)
d_i^{i+ \bar l-t+1}=0$ and $\lambda d_{i-t}^{i+1}=0$. The square
on the left hand side commutes if $\lambda d_i^{i+\bar l-t+1}=0$.
If idim$M >0$, then $d_i^{i+ \bar l-t+1} \not =0$. Therefore
$\lambda=1$. Then $d_{i-t}^{i+1}=0$ or idim$M=1$ and $d_i^{i+\bar
l-t+1}=0$ or pdim$M=0$, so condition (2) and (3) hold.

The bottom row is isomorphic to $\Omega^{-1}(M)$. As $A$ is a
Nakayama algebra, $\Omega^{-1}(M)$ is indecomposable. Therefore
the complex $\Omega^{-1}(M)$ is indecomposable.

\medskip

To prove the third part, we assume that there is a retraction $x$
from $\cone(w)$ to the given complex. The map $x$ and the map $y$
such that $y \circ x=\id$ can be taken as
\[ \xymatrix{ \cdots \ar[r] & I_{i-l} \ar[d] \ar[r]_{\binom{d}{0}}& I_{i-t+1}
\oplus I_{i-t} \ar[d]_{(\id,\lambda d)}
\ar[r]_{\binom{-d,0}{d,d}} & I_{i+1} \oplus I_{i}
\ar[r]_{(d,0) } \ar[d]_{(0,\id)} & I_{i+ \bar l -t+1} \ar[d] \ar[r] & \cdots \\
\cdots \ar[r]  & 0 \ar[r] \ar@/_/[u]&  I_{i-t+1}
\ar@/_/[u]_{\binom{\id}{0}} \ar[r]_d & I_{i} \ar[r]
\ar@/_/[u]_{\binom{\delta d}{\id} } & 0 \ar@/_/[u] \ar[r] &
\cdots}\] for some $\lambda, \delta \in k$. The square on the left
hand side of the map going down commutes if and only if $\lambda
d_{i-l}^{i-t+1}=0$. The middle square commutes if and only if
$(\lambda-1)d_{i-t}^i=0$. If pdim$M > 0$, then $d_{i-t}^i\not
=0$, so $\lambda =-1$ and we get that condition (1) is satisfied.

The square on the right hand side of the map going up commutes if
and only if $(\delta +1) d_{i-t+1}^{i+1} =0$ holds. The middle
square of the map going up commutes if and only if $\delta d_i^{i+
\bar l-t+1}=0$. If idim$M>0$, then $d_{i-t+1}^{i+1} \not =0$, so
$\delta=-1$ and $d_i^{i+ \bar l-t+1}=0$ or idim$M=1$. So
condition (2) holds.

Therefore the complex in the bottom row is a direct summand of
$\cone(w)$ if and only if (1) and (2) hold. The complex in the
bottom row is indecomposable by \ref{indecomposable}. This proves
part (a)-(d).

\smallskip

The complexes $\sigma^{\ge 1}(\cone(w))$ and $\sigma^{\le
-1}(\cone(w))$ are indecomposable in $Comp^{+,b}(\mathcal{I})$, as
the first one is a minimal injective resolution of an
indecomposable module and the second one is $\nu$ applied to a
minimal projective resolution of an indecomposable module.
Therefore the retractions presented in the previous three diagrams
are the only possibilities for direct summands. So if at most one
condition is satisfied, then $\cone(w)$ has to be indecomposable,
which proves part (e).

\end{proof}

We determine the number of predecessors of the simple $A$-modules
in the Auslander-Reiten quiver of $D^b(A)$.
\begin{lemma}\label{simple pre}
Let $S_i$ be a simple $A$-module and assume that $S_i$ is not projective and not injective. Then
$S_i$ has two predecessors in $\Gamma(D^b(A))$if and only if $d_{i-1}^{i+1}=0$. Otherwise $S_i$ has only one predecessor.
\end{lemma}
\begin{proof}
With the notation of \ref{pre}, we have that $t=1$ and $\bar
l=l(I_i)$. Clearly condition $(1)$ and condition $(2)$ are always
satisfied. Therefore $S_i$ has two predecessors if and only if (3)
is satisfied. This is the case if and only if $d_{i-1}^{i+1}=0$.
Then $S_i$ has the two predecessor $\nu (\rad P_i)$ and
$I_i/S_i[-1]$. In all other cases $S_i$ has only one
predecessor.
\end{proof}
We compute an example.
\begin{exa}
Let $ 1\stackrel{\alpha} \to 2\stackrel{\beta} \to 3
\stackrel{\gamma} \to 4$ and $I=\langle \alpha \beta \gamma \rangle$. Then a
complete list of indecomposable left $A$-modules is given as
$P_1=S_1=\rad P_2$, $P_2=\rad P_3$, $P_3= I_1$, $M:= P_3/ P_1=\rad
P_4$, $P_4= I_2$, $S_2$, $S_3$, , $I_3=P_4/S_2$ and $I_4=S_4$.

As $P_1$ is simple, its Auslander-Reiten triangle is of the form

\[ P_3[-1] \to M[-1] \to P_1 \to P_3. \]

As $I_2=P_4$ is projective, we have

\[ P_4[-1] \to \nu^{-1}(i(M)) \to P_2 \to P_4. \]

Therefore the Auslander-Reiten triangles of $P_3$ and $P_4$ are of the form

\[ I_3[-1] \to \nu (p(M)) \to P_3 \to I_3  \mbox{ and} \]

\[ I_4[-1] \to M \to P_4 \to I_4. \]

By \ref{pre} the module $M$ satisfies the conditions $(1)$ and
$(2)$ and therefore its Auslander-Reiten-triangle is of the form
\[ \nu (M)[-1] \to I_4[-1] \oplus P_3 \oplus S_2 \to M. \]

The simple modules $S_2$ and $S_3$ have only one predecessor by
\ref{simple pre}. Their Auslander-Reiten triangles are of the form
\[ S_2 \to M \to S_3 \to S_2[1] \]

and \[\nu^{-1}(S_3)\to \nu (M)[-1] \to S_2 \to \nu^{-1} (S_3)[1]. \] The $\tau$-orbits
are the following:

1st orbit: $P_1$, $ P_3[-1]$, $ I_3[-2]$, $ P_1[-1]$

2nd orbit: $ P_4$, $I_4[-1]$, $ P_2 $, $ P_4[-1] $

3rd orbit: $ M$, $\nu(M)[-1] $, $ \nu^{-1}(M)$, $M[-1]$

4th orbit: $ S_3$, $S_2$, $\nu (S_2)[-1]$, $S_3[-1]$.

So all elements of $A$-mod lie in one component $\Z[D_4]$
where $D_4$ is the following quiver:
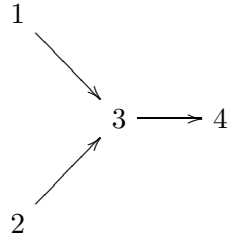
\[\xymatrix{ 1 \ar[rd] &&\\
& 3 \ar[r] & 4&\\
2 \ar[ru] & & }\] The number of the vertices of $D_4$ correspond
to the number of the $\tau$-orbits and represent their position in
the tree of the Auslander-Reiten component. Note also that the
shift $[-1]$ acts on the component as $\tau^4$.

By \cite[bounded]{S} it is clear that all indecomposable elements of
$K^b(\mathcal{P})$ belong to this component.
\end{exa}
We can determine the Auslander-Reiten quiver for some classes of Nakayama algebras.
\begin{theo}
Let $A$ be of global dimension $n-1$. Then the Auslander-Reiten quiver
is isomorphic to $\Z [ A_n]$.
In particular $[-1]$ is an involution on the tree class $A_n$ of $\Gamma(D^b(A))$. The algebra $A$ is
derived equivalent to $kA_n$.
\end{theo}
\begin{proof}
Let $ 1<i <n$, then $S_i$ is non-projective and non-injective.
Furthermore by \ref{simple pre} the simple module $S_i$ has the
two predecessors $\tau(S_{i-1})[1]$ and $S_{i+1}[-1]$.  We also
know by  by \ref{pro pre} that $S_1$ has the only predecessor
$S_2[-1]$ and $S_n$ has the
 only predecessor $\tau(S_{n-1})[1]$. So every complex in this
component lies in the $\tau$-orbit of a stalk complex of a simple
module up to shift.

Next we show that all complexes in the component are shift periodic. By direct computation we have
that $\tau^s(S_i)$ is isomorphic to
the complex \[ \cdots \to 0 \to I_s \to \cdots \to I_{i+s-1} \to 0
\to \cdots[i-s-1].\]Therefore $\tau^s(S_i)$ has two non-zero
homologies except for $s=n-i+1$ where
$\tau^{n-i+1}(S_i)=S_{n-i+1}[-n+2i-2]$ and $s=n+1$ where
$\tau^{n+1}(S_i)=S_i[-2]$. The $\tau $-orbit of $S_1$ is given by
$ \tau^s(S_1)= I_s[-s]$. Therefore a shift of all projective
indecomposable modules and of all injective indecomposable modules are in
the orbit of $S_1$. Also $\tau^n(S_1)=S_n[-n]$ and
$\tau(S_n)=S_1[n-2]$. Therefore $\tau^{n+1}(S_i)=S_i[-2]$ for all
$1 \le i \le n$. The above computation also shows that the $\tau$-orbits of non-isomorphic simple modules are distinct.
Therefore a maximal sectional
path is given by

\[ \xymatrix{ S_n[-n+1] \ar[r] & \cdots \ar[r]& S_2[-1] \ar[r] & S_1.}\]

The Auslander-Reiten component is therefore isomorphic to $\Z[A_n]$ given as above.
As the $[2]$ shift operates on $\tau$-orbits, we can see that all
elements in the component are of bounded length. By \cite[4.9]{S}
all indecomposable elements are part of the component and by
\cite[4.15]{S} the algebra $A$ is derived equivalent to $kA_n$.
\end{proof}
We investigate another class of Nakayama algebras.
\begin{theo}
Let $A:=k A_n/I$ with $n\ge 4$ and $I$ is generated by the path of
length $n-1$. Then the Auslander-Reiten quiver of $D^b(A)$ is
isomorphic to $\Z[D_n]$. If $n$ is even, then $[-1]$ acts as the
identity on $D_n$. If $n$ is odd $[-1]$ acts as the involution on
$D_n$. Furthermore $A$ is derived equivalent to $kD_n$.
\end{theo}
\begin{proof}
Let $1 < i< n$, then the projective resolution of $S_i$ is given
by $0 \to P_{i-1} \to P_i$ and the injective resolution by $I_i
\to I_{i+1}\to 0$. Therefore $\tau(S_i)= S_{i-1} $ for $2 <i< n$,
$\tau(S_2) = \cdots \to 0 \to I_1 \to I_2 \to 0 \to \cdots$ and
$\tau^2(S_2)=S_{n-1}[-1]$. All non-injective and
non-projective simple modules are in the same $\tau$-orbit and $[-1]$
operates on that orbit. By \ref{simple pre} each $S_i$ has exactly
one predecessor and it is of the form \[ \cdots \to 0 \to I_{i-1} \to I_{i+1}
\to 0 \to \cdots \] which is isomorphic to the stalk complex with stalk $P_i/\rad^2(P_i)$ for $
i>2$.

We note that $I_i$ has projective resolution $0\to P_1 \to P_{i-1}
\to P_n$ for $i>2$. Therefore $\tau(I_i)=P_{i-2}[1]$ for $i>2$,
$\tau(I_2)=I_n[-1]$ and $\tau(I_1)=I_{n-1}[-1]$.

Suppose now that $n$ is even, then the orbit of $S_1=P_1$ is given
by $P_1,$ $I_1[-1],$ $I_{n-1}[-2],$ $P_{n-3}[-1],$ $ \cdots,$
$I_{n-(2k+1)}[-2],$ $ P_{n-(2k+3)}[-1],$ $ \cdots,$ $P_1[-1]$. In this case
the orbit of $S_n=I_n$ is given by $S_n,$ $P_{n-2}[1]$, $I_{n-2},$
$\cdots,$ $I_{n-2k},$ $P_{n-2k-2}[1],$ $\cdots,$ $I_2,$ $S_n[-1]$.

Suppose now that $n $ is odd. Then the orbit of $P_1$ is given by
$P_1,$ $I_1[-1],$ $I_{n-1}[-2],$ $ P_{n-3}[-1],$ $\cdots,$
$I_{n-(2k+1)}[-2],$ $P_{n-(2k+3)}[-1],$ $ \cdots,$ $I_2[-2],$ $S_n[-3],$
$P_{n-2}[-2]$, $I_{n-2}[-3]$, $\cdots,$ $I_{n-2k}[-3],$ $P_{n-2k-2}[-2],$
$\cdots,$ $P_1[-2]$ and contains the odd shifts of $S_n$.

So $S_n$ and $S_1$ lie in different orbits. If $n$ is even [-1]
operates on each orbit. If $n$ is odd, $[-2]$ operates on the
orbits and [-1] maps the orbit of $S_1$ onto the orbit of $S_n$.

The only predecessor of $S_1$ is given by
$P_{n-1}/\rad^{n-2}(P_{n-1})[-1]$.

Next we investigate the predecessors of modules
$M_s:=P_s/\rad^{s-1}(P_s)$ for $n-1 \ge  s \ge 2$. The projective
resolution of $M_s$ is given by $0 \to P_1 \to P_s$ and the
injective resolution of $M_s$ is given by $I_2 \to I_{s+1}\to 0$.
Therefore (1) and (2) are always satisfied. We have $d_1^{s+1} =0$
if and only if $s=n-1$. So $M_s$ has three predecessor if and only
if $s=n-1$ and else it has only two predecessor. By the proof of
\ref{pre}, the predecessors of $M_{n-1}$ are $S_n[-1]$,
$\tau(S_1[1])$ and $M_{n-2}$. The predecessors of $M_s$ for $2< s<
n-1$ are $M_{s-1}$ and $\tau^{-1}( M_{s+1})$. Therefore a maximal
sectional path of the Auslander-Reiten component looks as follows

\[\xymatrix{  & & & & S_1[1] \\
S_2\ar[r] & M_3 \ar[r] & \cdots \ar[r] & M_{n-1}
\ar[ur]  & \\
& & & & S_n \ar[lu] [-1] } \]

The component is isomorphic to $ \Z[ D_n]$.
Furthermore we know that $[-2]$ acts on the $\tau$-orbit of $S_i$,
$S_n$ and $S_1$. Therefore $[2]$ acts as the identity on all
$\tau$-orbits of the component. By \cite[4.9]{S} this component is
the only Auslander-Reiten component of $D^b(A)$ and by
\cite[4.15]{S} the algebra $A$ is derived equivalent to $kA_n$.
\end{proof}


\begin{thebibliography}{Mon}
\bibitem[ASS]{ASS}I.\ Assem, D.\ Simson and A.\ Skowro\'nski, \emph{Elements of the representation theory of associative algebras}, London Mathematical Society Student Texts, 65.\ Cambridge Univ.\ Press, Cambridge, 2006.
\bibitem[BGS]{BGS} G.\ Bobiski, C.\ Gei, A.\ Skowroski, \emph{Classification of discrete derived categories}, Cent.\  Eur.\ J.\ Math.\  2  (2004),  no.\ 1, 19-49.
\bibitem[H1]{H1} D.\ Happel, \emph{Triangulated Categories in the representation theory of finite-dimensional algebras}, London Mathematical Society Lecture Note Series, 119.\ Cambridge University Press, Cambridge, 1988.
\bibitem[H2]{H} D.\ Happel, \emph{ Auslander-Reiten Triangles in Derived categories of Finite-Dimensional Algebras}, Proc.\ Amer.\ Math.\ Soc.\  112  (1991),  no.\ 3, 641-648.
\bibitem[HPR]{HPR} D.\ Happel, U.\ Preiser and C.\ M.\ Ringel, \emph{Vinberg's characterization of Dynkin diagrams using
subadditive functions with application to $D{\rm Tr}$-periodic
modules}. Representation Theory II,\ Lecture Notes in Math.\, 832
(1980), Springer, Berlin, 1980, 280-294.
\bibitem[Ri]{Rie} C.\ Riedtmann, \emph{Algebren, Darstellungsköcher, Überlagerungen und zurück}, Comment.\ Math.\ Helv.\  55  (1980), no.\ 2, 199-224.
\bibitem[S]{S} S.\ Scherotzke, \emph{Finite and bounded Auslander-Reiten components in the derived category}, archiv:   . 
\bibitem[We]{Web}P.\ Webb, \emph{The Auslander-Reiten quiver of a finite group}.
Math.\ Z.\ 179 (1982), no.\ 1, 97-121.
\bibitem[Wei]{We}C.\ Weibel, \emph{ An introduction to homological algebra}, Cambridge Studies in Advanced Mathematics, 38.\ Cambridge University Press, Cambridge, 1994.
\bibitem[W]{W}W.\ Wheeler, \emph{ The triangulated structure of the stable derived category},
J.\ Algebra 165 (1994), no.\ 1, 23-40.
\bibitem[V]{V} D.\ Vossieck, \emph{ The algebras with discrete derived
category} J.\ Algebra 243 (2001), no.\ 1, 168-176.
\bibitem[XZ]{XZ} J.\ Xiao, B.\ Zhu, \emph{ Locally finite triangulated categories}, J.\ Algebra 290 (2005), no.\ 2, 473-490.
\end{thebibliography}
\end{document}